\documentclass[11pt]{amsart}
\usepackage{amsmath,amsthm,amssymb,xypic,array}
\usepackage[T1]{fontenc}
\usepackage{amsfonts}
\usepackage{amsmath}
\usepackage{amssymb}
\usepackage{amsthm}
\usepackage{setspace}
\usepackage[cp850]{inputenc}
\usepackage[bookmarksnumbered,plainpages,hypertex]{hyperref}
\usepackage{graphicx}
\usepackage{fancyhdr}
\usepackage{tikz}

\providecommand{\U}[1]{\protect\rule{.1in}{.1in}}
\providecommand{\U}[1]{\protect\rule{.1in}{.1in}}
\providecommand{\U}[1]{\protect\rule{.1in}{.1in}}
\providecommand{\U}[1]{\protect\rule{.1in}{.1in}}
\providecommand{\U}[1]{\protect\rule{.1in}{.1in}}
\input{xy}
\xyoption{all}
\theoremstyle{theorem}

\newtheorem{Theorem}{Theorem}[section]
\newtheorem{theoremn}{Theorem}

\newtheorem{Proposition}[Theorem]{Proposition}
\newtheorem{Corollary}[Theorem]{Corollary}
\newtheorem{Conjecture}[Theorem]{Conjecture}

\theoremstyle{definition}

\newtheorem{Definition}[Theorem]{Definition}
\newtheorem{Remark}[Theorem]{Remark}

\newtheorem{Example}[Theorem]{Example}

\numberwithin{equation}{section}

\newcommand{\arXiv}[1]{\href{http://arxiv.org/abs/#1}{arXiv:#1}}

\newcommand{\Pic}{\operatorname{Pic}}

\newcommand{\Res}{\operatorname{Res}}

\newcommand{\p}{{\mathbb P}}
\renewcommand{\a}{{\`a}}

\newcommand{\Mor}{\operatorname{Mor}}

\def\leq{\leqslant}
\def\geq{\geqslant}

\def\bibaut#1{{\sc #1}}

\def\phi{\varphi}
\def\ro[#1]{{\textcolor{red}{#1}}}

\renewcommand{\a}{\`a }

\begin{document}

\begin{abstract}
We study some properties of an embedded variety covered by lines and give a numerical criterion ensuring the existence of a singular conic through two of its general points. We show that our criterion is sharp.
Conic-connected, covered by lines, $QEL$, $LQEL$, prime Fano, defective, and dual defective varieties are closely related. We study some relations between the above mentioned classes of objects using celebrated results by Ein and Zak. 
\end{abstract}

\title{Covered by lines and Conic connected varieties}
\author[Simone Marchesi]{Simone Marchesi}
\address{\sc Simone Marchesi\\
Dipartimento di Matematica "Federigo Enriques"\\
Universit\a degli Studi di Milano\\
via Cesare Saldini  50\\
20133 Milano\\
Italy}
\address{Departamento de \'{A}lgebra\\
Facultad de Ciencias Matem\'{a}ticas\\
Universidad Complutense de Madrid\\
Plaza de las Ciencias  3\\
28040 Madrid\\
Spain}
\email{simone.marchesi@unimi.it, smarches@mat.ucm.es}
\author[Alex Massarenti]{Alex Massarenti}
\address{\sc Alex Massarenti\\
SISSA\\
via Bonomea 265\\
34136 Trieste\\ Italy}
\email{alex.massarenti@sissa.it}
\author[Saeed Tafazolian]{Saeed Tafazolian}
\address{\sc Saeed Tafazolian\\
School of Mathematics, Institute for Research in Fundamental Sciences (IPM)\\
P. O. Box: 19395-5746\\
Tehran\\
Iran}
\address{Institute for Advanced Studies in Basic Sciences (IASBS)\\
Department of Mathematics\\
P.O. Box 45195-1159\\
Zanjan\\
Iran}
\email{tafazolian@iasbs.ac.ir}
\address{}
\email{}

\maketitle

\tableofcontents

\section*{Introduction}
The study of rational curves on algebraic varieties is of fundamental importance in algebraic geometry. Indeed the birational geometry of a smooth projective variety is closely related to the rational curves it contains. Many tools have been introduced for this purpose, such as \textit{Mori theory} (see \cite{Mori}) which has been a great breakthrough in the theory of minimal models.

These issues naturally lead to the study of varieties covered by rational curves and rationally connected varieties. Over an algebraically closed field of characteristic zero a variety $X$ is said to be \textit{uniruled} if for $x\in X$ general point there exists a rational curve on $X$ through $x$, while $X$ is \textit{rationally connected} if two general points $x,y\in X$ can be connected by a rational curve on~$X$.

This subject can be explored in an abstract context, for instance by \textit{Hwang} and \textit{Kebekus} in \cite{HK}, or in an embedded setting, by \textit{Ionescu} and \textit{Russo} in \cite{IR2}.
An intermediate point of view is to consider varieties polarized by ample divisors, for instance by \textit{Lanteri} and \textit{Palleschi} in \cite{LP}.

We shall place ourselves in an embedded context, that is considering the variety $X$ as a subvariety of a projective space $\mathbb{P}^{N}$ and using techniques coming from classical algebraic geometry. From this point of view the so called variety of minimal rational tangents of $X$ at a point $x$ (see \cite{Hwang}) is simply the variety $\mathcal{L}_{x}$ parameterizing lines contained in $X$ through $x$. The simplest case of rational connectedness  to study, in the embedded setting, is the existence of a line through two generic points; clearly this is not interesting because a variety with this property is necessarily a linear space. A more interesting case is given by considering the next one, i.e. when two general points can be connected by a conic; varieties with this property are called \textit{conic-connected}, $CC$ for short. Such property is mostly studied in the context of covered by lines, secant defective, $QEL$, and $LQEL$ varieties.

Consider a smooth irreducible $n$-dimensional complex variety $X\subset\mathbb{P}^{N}$ (we will denote by $c$ its codimension), its secant variety $SX$ is the closure of the locus of its secant lines. The \textit{secant defect} of $X$ is the number $\delta(X) = 2n + 1 - \dim(SX)$; the variety is called \textit{secant defective} if $\delta(X)\geq 1$. The locus determined on $X$ by the cone of secant lines through a general point $z\in SX$ is called the \textit{entry locus} of $X$ and denoted by $E_{z}$, note that $E_{z}$ is a purely $\delta$-dimensional subvariety of $X$. The variety $X$ is said to be $QEL$ (\textit{quadratic entry locus}) if $E_{z}$ is a quadric, while $X$ is a $LQEL$ variety (\textit{local quadratic entry locus}) if for $x,y\in X$ general points there is a quadric $Q_{x,y}\subset X$ through $x,y$. These classes of varieties have been widely studied by \textit{Ionescu} and \textit{Russo} in \cite{Ru}, \cite{IR3} and \cite{IR2}.

In section \ref{NP} we will give preliminary notions and some notation and in section \ref{EZ} we will recall two basic theorems due to Zak and Ein, see \cite{Zak} and \cite{Ein}.

In section \ref{PF} we concentrate on non trivial relations between these classes of varieties, dual defective and prime Fano varieties, mainly using \textit{Zak's Theorem on tangencies}, \textit{Ein's classification of dual defective varieties} and properties of $\mathcal{L}_{x}$. Indeed $X$ inherits significant properties from the geometry of $\mathcal{L}_{x}$; see \cite{IR} for a discussion on this issue. In particular we show that if $a\geq n-c$ holds, where $a:=\dim(\mathcal{L}_{x})$, we also have $a\leq \frac{n+c-3}{2}$. Moreover if the last bound is an equality, well known varieties naturally arise such as the Grassmannian $\mathbb{G}(1,4)\subset\mathbb{P}^{9}$ (lines in $\p^{4}$), and the Spinor variety $\mathbb{S}^{10}\subset\mathbb{P}^{15}$. We highlight a relation between varieties covered by lines such that $a\geq n-c$ and the \textit{Hartshorne conjecture on complete intersections} (Conjecture \ref{hconj}). 
In section \ref{PS} we study conics on $X$ and get the following result (Theorem~\ref{CC}).
\begin{theoremn}
Let $X\subset\p^{N}$ be a variety set theoretically defined by homogeneous polynomials $G_{i}$ of degree $d_{i}$, for $i = 1,\ldots,m$. If 
$$\sum_{i=1}^{m}d_{i}\leq \frac{N+m}{2},$$
then $X$ is connected by singular conics.

Assume $X$ to be smooth and the equations $G_{i}$'s to be scheme theoretical equations for $X$ and in decreasing order of degrees. If 
$$\sum_{i=1}^{c}d_{i}\leq \frac{N+c}{2},$$
where $c=N-n$, then $X$ is conic-connected by smooth conics also.
\end{theoremn}
This result is closely related to a result obtained by \textit{Bonavero} and \textit{H\"oring} in \cite{BH}, which gives a numerical criterion for conic-connectedness. However while \textit{Bonavero} and \textit{H\"oring} only consider schematic smooth complete intersections we allow $X$ to be singular and give a condition ensuring the existence of a singular conic through two general points. Furthermore, in Remark \ref{SHA}, we show that our inequality is sharp considering a smooth cubic hypersurface in~$\p^{4}$.

\section{Notation and Preliminaries}\label{NP}
We work over the complex field. We mainly follow notation and definitions of \cite{IR}. Throughout this paper we denote by $X\subset \p^{N}$ a smooth irreducible variety of dimension $n\geq 1$. We assume $X$ to be non-degenerate of codimension $c$, so that $N = n+c$. If $x\in X$, we write $T_xX$ for the projective closure of the embedded Zariski tangent space of $X$ at $x$.

\subsubsection*{The Secant Variety}
Let $X\subset \p^N$ be a closed, irreducible subvariety of dimension $n$. Consider the following incidence variety, $\mathcal{S}_X$, called the {\it abstract secant variety} of $X$:
$$\mathcal{S}_X =\overline{\{(x,x',t)|\  x,x'\in X, x\neq x', t\in \langle x,x'\rangle\subset \p^N\}}\subset X\times X\times \p^N,$$
with $\mathcal{S}_X$ irreducible, of dimension $2n+1$.
\begin{Definition} Let $X\subset \p^N$ be an irreducible variety. Its {\it secant variety}, denoted by $SX$, is the image of $\mathcal{S}_X$ in $\p^N$ via the natural projection.
The dimension of $SX$ may be smaller than $2n+1$. In this case we say that $X$ is {\it secant defective} and introduce the {\it secant defect} of $X$ to be $\delta:=2n+1-\dim(SX)\geq 0$.
\end{Definition}
As $SX\subseteq \p^N$, we have that $\dim(SX)\leq N$, which implies $\delta\geq n-c+1$, where $c$ is the codimension of $X$ in $\p^N$.

\subsubsection*{QEL, LQEL, and CC Varieties}
Let $x,y\in X$ be two general points, and let $z \in l_{x,y}$ be a general point on the secant line $l_{x,y} = \langle x,y\rangle$. The trace on $X$ of the closure of the locus of secants to $X$ passing through $z$ is called the \textit{entry locus} of $X$ with respect to $z$, denoted by $E_{z}$. We have that $\dim(E_{z}) = \delta = 2n+1 - \dim(SX)$.
\begin{Definition}
A secant defective variety $X\subset \p^{N}$ of secant defect $\delta$ is called a \textit{quadratic entry locus} $(QEL)$ variety if $E_{z} = Q^{\delta}$ is a $\delta$-dimensional quadric. It is called a \textit{local quadratic entry locus} $(LQEL)$ variety if for $x,y\in X$ general points there exists a $\delta$-dimensional quadric passing through $x,y$ and contained in $X$.
\end{Definition}
Finally we are interested in the first non trivial example of rational connectedness in the embedded case, which occurs when two general points are connected by a conic.
\begin{Definition}
A variety $X\subset \p^{N}$ is called \textit{conic-connected} $(CC)$ variety if for $x,y \in X$ general points there is a conic $C_{x,y}$ passing through $x,y$ and contained in~$X$.
\end{Definition}
If $X$ is $QEL$ and $x,y\in X$ are general points, we can consider a point $z$ on the secant line $l_{x,y}$; the entry locus $E_{z}$ is a $\delta$-dimensional quadric through $x,y$, so $X$ is $LQEL$. Furthermore if $X$ is $LQEL$ and $Q^{\delta}_{x,y}$ is a quadric such that $x,y\in Q^{\delta}_{x,y}\subseteq X$, we can take a plane section $\Pi_{x,y}\cap Q^{\delta}_{x,y} = C_{x,y}$, where $\Pi_{x,y}$ is a plane containing $x,y$. Clearly $C_{x,y}$ is a conic such that $x,y\in C_{x,y}\subseteq X$. To summarize what we said
$$QEL \Longrightarrow LQEL \Longrightarrow CC.$$
For an extensive discussion on $QEL$, $LQEL$, and $CC$ varieties see \cite{Ru}, \cite{IR2}, and \cite{IR3}. 

\subsubsection*{Dual Defective Varieties}
Given a variety $X\subset \p^N$ as above, let us consider the conormal variety
       \[
  \begin{tikzpicture}[xscale=1.5,yscale=-1.2]
    \node (A0_1) at (1, 0) {$\mathcal{C} = \{(x,H) \: | \: T_{x}X\subseteq H\}\subseteq X\times (\p^{N})^{*}$};
    \node (A1_0) at (0, 1) {$X$};
    \node (A1_2) at (2, 1) {$(\mathbb{P}^{N})^{*}$};
    \path (A0_1) edge [->]node [auto] {$\scriptstyle{\pi_{2}}$} (A1_2);
    \path (A0_1) edge [->]node [auto,swap] {$\scriptstyle{\pi_{1}}$} (A1_0);
  \end{tikzpicture}
  \]
Clearly the map $\pi_{1}$ is surjective and its fibers are linear spaces of dimension $c-1$. So $\dim(\mathcal{C}) = n+c-1 = N-1$. Let $H\in (\p^{N})^{*}$ be a hyperplane. The fiber $\pi_{2}^{-1}(H)$ consists of the couples $(x,H)$ such that $H\supseteq T_{x}X$, i.e. $H$ is a contact hyperplane at $x$ to $X$. The image $\pi_{2}(\mathcal{C}) = X^{*}\subseteq (\p^{N})^{*}$ is called the \textit{dual variety} of $X$. Note that 
\begin{itemize}
\item[-] $\dim(X^{*})\leq N-1,$
\item[-] $\dim(X^{*})=N-1 \: \Longleftrightarrow \: \pi_{2}$ is generically finite.
\end{itemize}
\begin{Definition} If $k:= N-1-\dim(X^{*}) > 0$ then $X$ is called \textit{dual defective}, and $k$ is called the \textit{dual defect} of $X$. 
\end{Definition}

\section{Some results by Ein and Zak}\label{EZ}

In this section we recall two theorems by Zak and Ein respectively, which will be fundamental in the proofs of some of our results. For details and complete proofs we refer to \cite{Ein} and \cite{Zak}. The following result, due to Zak, gives a bound on the dimension of the singular locus of a linear section of $X$.

\begin{Theorem}(\underline{Zak's Theorem on tangencies})\label{ZT}
Let $X\subset \p^{N}$ be a non-degenerate variety of dimension $n$. Let $L$ be an $l$-dimensional linear space in $\p^{N}$ with $l\geq n$. Then
\begin{itemize}
\item[-] $\dim(Sing(L\cap X))\leq l-n$.
As a consequence,
\item[-] $\dim(X^{*})\geq \dim(X)$.
\end{itemize}
\end{Theorem}

It is natural to try to classify equality cases in the second inequality from Zak's Theorem above. A partial answer to this question is given by the following theorem by Ein. The answer is partial because condition $n \leq 2c$ is imposed. If the Hartshorne Conjecture, which we will recall later, holds, the condition $n \leq 2c$ would not be restrictive, since complete intersections are not dual defective.

\begin{Theorem} (\underline{Ein})\label{Ein}
Let $X\subset \p^{N}$ be a non-degenerate variety of dimension $n$ and codimension $c$, such that $n\leq 2c$. Suppose that $\dim(X^{*}) = \dim(X)$. Then one of the following holds:
\begin{itemize}
\item[-] $X$ is a hypersurface in $\p^{2}$ or $\p^{3}$;
\item[-] $X$ is projectively equivalent to the Segre variety $\p^{1}\times \p^{n-1}$ in $\p^{2n-1}$;
\item[-] $X$ is projectively equivalent to the Grassmannian $\mathbb{G}(1,4)$ in $\p^{9}$;
\item[-] $X$ is projectively equivalent to the $10$-dimensional Spinor variety $\mathbb{S}^{10}$ in $\p^{15}$.
\end{itemize}
\end{Theorem}

\section{Prime Fano varieties and varieties covered by lines}\label{PF}
Let $x \in X\subset \p^N$ be a general point. If $\mathcal{L}$ is an irreducible component of the Hilbert scheme of lines of $X$, we denote by $\mathcal{L}_{x}$ the variety of lines from $\mathcal{L}$ passing
through $x$. Note that $\mathcal{L}_{x}$ is embedded in the space of tangent directions at $x$, that is $\mathcal{L}_{x}\subseteq \p(t_xX^*)=\p^{n-1}$, where $t_xX$ denotes the affine embedded Zariski tangent space at $x$.

We denote by $a:= \dim(\mathcal{L}_{x})$. We say that $X$ is \textit{covered by the lines in} $\mathcal{L}$ if $\mathcal{L}_{x}\neq \emptyset$ for $x\in X$ general. It can be  proved that $a = \deg(\mathcal{N}_{l/X})$, where $l$ is a line from $\mathcal{L}$ and $\mathcal{N}_{l/X}$ is its normal bundle. When $a\geq \frac{n-1}{2}$, $\mathcal{L}_x\subset \p^{n-1}$ is smooth and irreducible; if, moreover, $\Pic(X)$ is cyclic, it is also non-degenerate, see \cite{Hwang}.
Recall that $X\subset \p^N$ is a \textit{prime Fano variety of index $i(X)$} if its Picard group is generated by the class $H$ of a hyperplane section and $-K_{X}=i(X)H$ for some positive integer $i(X)$. By the work of Mori, see \cite{Mori}, if we have $i(X) > \frac{n+1}{2}$ then $X$ is covered by a family of lines.

In this section we derive an inequality involving the parameters $n, c, a$, and then we classify the border cases.

Let us remark first that an interesting case is $a \geq n- c$, because in such situation each line from $\mathcal{L}$ is a contact line, which implies that the variety $X$ is not a complete intersection. Indeed, if $a \geq n-c$, then $\dim \left(\langle \bigcup_{x \in l} T_x X\rangle \right) \leq N -1$ for each line $l \subset X$, see \cite[Proposition 2.5]{IR}. 

\begin{Proposition}\label{primaprop}
Let $X\subset \p^N$ be a variety covered by an irreducible family of lines $\mathcal{L}$. If $a\geq n-c$, then $a\leq \frac{n+c-3}{2}$.
\end{Proposition}
\begin{proof}
Let $x\in X$ be a general point, and let $l\subseteq X$ be a line through $x$. We may consider the incidence variety
  \[
  \begin{tikzpicture}
    \def\x{1.5}
    \def\y{-1.2}
    \node (A0_1) at (1*\x, 0*\y) {$\mathcal{I} = \{(l,H)\: | \: H\supseteq \bigcup_{y\in l}T_{y}X\}\subseteq \mathcal{L}_{x}\times T_{x}X^{*}$};
    \node (A1_0) at (0*\x, 1*\y) {$\mathcal{L}_{x}$};
    \node (A1_2) at (2*\x, 1*\y) {$T_{x}X^{*}$};
    \path (A0_1) edge [->] node [auto] {$\scriptstyle{\pi_{2}}$} (A1_2);
    \path (A0_1) edge [->] node [auto,swap] {$\scriptstyle{\pi_{1}}$} (A1_0);
  \end{tikzpicture}
  \]
Let $H$ be a hyperplane in $\p^{N}$. By Zak's Theorem on tangencies, its contact locus is of dimension at most $c-1$. If $H\in Im(\pi_{2})$, the contact locus of $H$ contains the locus of lines from $\pi_{2}^{-1}(H)$. We get $\dim(\pi_{2}^{-1}(H))\leq c-2$. Furthermore, $\dim(Im(\pi_{2}))\leq \dim(T_{x}X^{*}) = c-1$, so
$$\dim(\mathcal{I}) \leq 2c-3.$$  
By \cite[Proposition 2.5]{IR} we know that $\dim \left(\langle \bigcup_{y \in l} T_y X\rangle \right)\leq 2n-a-1$ for $l\in \mathcal{L}_{x}$. It follows that 
$\dim((\pi_{1}^{-1})(l))\geq a-n+c$ for $l\in \mathcal{L}_{x}$. Since $a\geq n-c$, any line is contact and $\pi_{1}$ is surjective. Then $\dim(\mathcal{I})\geq 2a-n+c$, and combining the inequalities
$$2a-n+c \leq \dim(\mathcal{I}) \leq 2c-3,$$
we get $a\leq \frac{n+c-3}{2}$.
\end{proof}

The following remark gives some information about the case when the upper bound in the above proposition does not hold.

\begin{Remark}\label{primormk}
If $X\subset \p^N$ is covered by lines and $a \geq \frac{n+c-2}{2}$, then $X$ is $CC$ and $n\geq 3c$; indeed for two general points, the two cones made by the locus of lines of $X$ passing through those points intersect (being of dimension at least half the dimension of the ambient projective space). By Proposition \ref{primaprop}, we know that $a \leq n-c-1$, giving $n \geq 3c$.
\end{Remark}

As usual when one gets an inequality it is nice to classify cases for which equality holds.

\begin{Proposition}\label{secondaprop}
Let $X\subset \p^N$ be a variety covered by lines. Suppose $a\geq n-c$ and $a = \frac{n+c-3}{2}$. Then $X$ is dual defective and $\dim(X)=\dim(X^*)$. If in addition we assume $n\leq 2c$, then 
\begin{itemize}
\item[-] $X$ is projectively equivalent to the Segre embedding $\p^{1}\times \p^{n-1}$ in $\p^{2n-1}$, $n\geq 3$, or
\item[-] $X$ is projectively equivalent to the Grassmannian $\mathbb{G}(1,4)$ in $\p^{9}$, or
\item[-] $X$ is projectively equivalent to the Spinor variety $\mathbb{S}^{10}$ in $\p^{15}$.
\end{itemize}
\end{Proposition}
\begin{proof}
We refer to the previous proof and consider the inequalities $2a-n+c \leq \dim(\mathcal{I}) \leq 2c-3$. From $a = \frac{n+c-3}{2}$ we get $2a-n+c = 2c-3 = \dim(\mathcal{I})$. Furthermore $\dim(Im(\pi_{2})) \geq 2c-3 - (c-2) = c-1$, and since $\dim(Im(\pi_{2}))\leq \dim(T_{x}X^{*}) = c-1$ we get $\dim(Im(\pi_{2})) = c-1$, i.e. $\pi_{2}$ is surjective. This last fact means that the general hyperplane tangent at $x$ to $X$ is tangent along a positive dimensional subvariety; but $x\in X$ is a general point, so a general tangent hyperplane is tangent along a positive dimensional subvariety, in other words $X$ is dual defective. Moreover $\dim(\pi_{2}^{-1}(H)) = 2c-3-(c-1) = c-2$ for $H\in T_xX^*$ general, so the dual defect of $X$ is $k\geq (c-2)+1 = c-1$. Hence the dimension of the dual variety is $\dim(X^{*}) = N-1-k \leq n$.\\ 
On the other hand by Theorem \ref{ZT} we have $\dim(X^{*}) \geq \dim(X) = n$, and then $\dim(X^{*}) = \dim(X)$. Since $n\leq 2c$ the hypothesis of Theorem \ref{Ein} are satisfied. Note that $n = \dim(X^{*}) = N-1-k$, so $k = c-1$. Since $X$ is dual defective, it can not be a hypersurface. So we are left with the three cases listed in our proposition. It is easy to see that, conversely, all the varieties listed satisfy our hypotheses. To conclude, we remark that the conditions $a = \frac{n+c-3}{2}$ and $n\leq 2c$ imply that $a\geq n-c$, unless $c\leq 2$. The cases $c=1$ and $c=2$ lead to only one new case, the two-dimensional quadric in $\p^3$.
\end{proof}

We have already noticed that since $a\geq n-c$ the variety $X$ is not a complete intersection. Let us recall the \textit{Hartshorne Conjecture on complete intersections}. For details on this Conjecture we refer to \cite{Ha}. 
  
\begin{Conjecture}[Hartshorne]\label{hconj}
Let $X\subseteq\p^{N}$ be a non-degenerate smooth variety such that $n \geq 2c+1$. Then $X$ is a complete intersection.
\end{Conjecture}

Since complete intersections are not dual defective, the truth of the Hartshorne Conjecture would make the condition $n \leq 2c$ in the previous proposition superfluous. This argument also leads to the following conjecture:

\begin{Conjecture}
Let $X\subset \p^N$ be a non-degenerate variety covered by lines. If $a\geq n-c$, then $n\leq 2c$.
\end{Conjecture}

As an application of Proposition \ref{primaprop}, we have the following result, showing that prime Fano varieties of high index are quite special.

\begin{Proposition}\label{terzaprop}
Let $X\subset \p^N$ be a prime Fano variety covered by lines and let $i=i(X)$. If $i\geq \frac{n+\delta}{2}$ then one of the following holds:
\begin{enumerate}
\item[-] $X$ is a quadric, or\label{case1}
\item[-] $c\geq 3$ and $n\leq 2c$\label{case2}. Moreover:
\begin{itemize}
\item[-] if $X$ is a $CC$ variety, then $X$ is an $LQEL$ variety;
\item[-] if $n=2c$, then $X \simeq \mathbb{G}(1,4) \subset \mathbb{P}^9$ or $X \simeq \mathbb{S}^{10} \subset \mathbb{P}^{15}$.
\end{itemize}
\end{enumerate}
\end{Proposition}
\begin{proof}
We know that for a prime Fano variety covered by lines we have $i=a+2$, so our hypothesis becomes $a\geq \frac{n +\delta -4}{2}$. Recalling that $\delta \geq n-c+1$, we get $a \geq \frac{2n-c-3}{2}$. We would like to use Proposition \ref{primaprop}. We have $a \geq \frac{2n-c-3}{2}$ and this is also greater or equal to $n-c$, unless $c \leq 2$.

Let us suppose $n > 2c$ and consider the case $c\geq 3$. We know by Proposition \ref{primaprop} that $a \leq \frac{n+c-3}{2}$ and considering both inequalities, we get
$$
\frac{n+c-3}{2} \geq a \geq \frac{2n-c-3}{2},
$$
which gives us $n\leq 2c$ and a contradiction. Note also that the equality $n=2c$  forces that $a=\frac{n+c-3}{2}$.

If we suppose $c=2$, we have that $i \geq \frac{2n-1}{2}-\frac{1}{2}$, giving us two possibilities:
\begin{itemize}
\item[-] $i$, so $X\simeq Q^n \subset \mathbb{P}^{n+2}$. This leads to contradiction because we supposed $X$ to be non-degenerate.
\item[-] $i+1$, so $X\simeq \mathbb{P}^n \subset \mathbb{P}^{n+2}$. This also leads to contradiction for the same reason.
\end{itemize}
Let us finally take $c=1$, where we have $i \geq \frac{2n-c+1}{2}=\frac{2n}{2} = n$. Also in this case the two possibilities are $i+1$, that gives us a hyperplane of the projective space $\p^N$, which is of course degenerate and leads to contradiction; or $i$, where $X \simeq Q^n \subset \mathbb{P}^{n+1}$. This is the only possible case and we have completed the first half of the classification.

Assume now that $n\leq 2c$ and $c\geq 3$. If $X$ is $CC$, using \cite[Proposition 3.2]{IR3} we observe that $n+1 \leq -K_X\cdot C \leq n +\delta$, where $C$ is a conic contained in $X$ and passing through two general points. We have $-K_X\cdot C = 2i$, so $\frac{n+1}{2} \leq i \leq \frac{n+\delta}{2}$. In our case we find the equality $i = \frac{n+ \delta}{2}$, whose consequence is that $X$ is an $LQEL$ variety by \cite[Proposition 3.2]{IR3}. Let us consider, finally, the boundary case $n=2c$. As we already remarked above, this forces the equality  $\frac{n+c-3}{2}= a$ and we may apply Proposition \ref{secondaprop}. Note that the first case is excluded since it is not a prime Fano variety, while the other two satisfy our assumptions.
\end{proof}

\section{Some results on Conic Connectedness}\label{PS}
In this section $X\subset\p^{N}$ will be a variety of dimension $n$ set theoretically defined by $G_{1},\ldots,G_{m}$, where $G_{i}\in k[x_{0},\ldots,x_{N}]_{d_{i}}$ is a homogeneous polynomial of degree $d_{i}$. Our aim is to give a relation between the parameters $n, c, m$ and $d_{i},$ ensuring the conic connectedness of $X$. 

\subsubsection*{Spaces of Morphisms}
In this part we mainly follow the notation and the definitions of \cite{De}. We introduce spaces of morphisms because we will refer to these in Remark \ref{Mori}. Let $f:\p^{1}\rightarrow\p^{N}$ be a morphism of degree $l$. We can write 
$$f(u,v) = [f_{0}(u,v):\cdots:f_{N}(u,v)],$$ 
where $f_{i}(u,v)\in k[u,v]_{l}$, and $f_{0},...,f_{N}$ do not have non constant common factors, so morphisms of degree $l$ from $\p^{1}$ to $\p^{N}$ are parameterized by a Zariski open subset of the projective space $\p((k[u,v]_{l})^{N+1})$ denoted by $\Mor_{l}(\p^{1},\p^{N})$. This open subset is the complement of the closed subset parameterizing polynomials $f_{0},...,f_{N}$ having common zeros. Note that these polynomials have a common zero if and only if $\Res(\sum_{i}z_{i}f_{i},\sum_{i}w_{i}f_{i})$ is identically zero as a polynomial in the $z_{i},w_{i}$. 
Morphisms of degree $l$ from $\p^{1}$ to $X$ are parameterized by the subscheme $\Mor_{l}(\p^{1},X)$ of $\Mor_{l}(\p^{1},\p^{N})$ defined by $G_{i}(f_{0},\cdots,f_{N}) = 0$ for $i = 1,...,m$. We denote by $\overline{\Mor_{l}(\p^{1},X)}$ its closure. Note that we are parameterizing morphisms from $\p^{1}$ to $X$ and not their images. Indeed the Chow scheme of rational curves of degree $l$ in $X$ can be obtained from an open subscheme of $\Mor_{l}(\p^{1},X)$ by taking the quotient by the action of $\rm Aut(\p^{1})$.

Let $[f]\in \Mor_{l}(\p^{1},X)$ be a morphism. By deformation theory of rational curves it is well known that the obstruction to deform $f$ lies in $H^{1}(\mathbb{P}^{1},f^{*}T_{X})$.\\ 
If $H^{1}(\mathbb{P}^{1},f^{*}T_{X}) = 0$ then $\Mor_{l}(\p^{1},X)$ is smooth at $[f]$ and its tangent space is given by $T_{[f]}\Mor_{l}(\p^{1},X)\cong H^{0}(\mathbb{P}^{1},f^{*}T_{X})$.

Fix now a subscheme $B\subseteq\mathbb{P}^{1}$, we denote by $\Mor_{l}(\p^{1},X,B)$ the scheme parametrizing morphisms fixing $B$. In this case the obstruction lies in $H^{1}(\mathbb{P}^{1},f^{*}T_{X}\otimes \mathcal{I}_{B})$ and  the irreducible components of $\Mor_{l}(\p^{1},X,B)$ are all of dimension at least
$$h^{0}(\mathbb{P}^{1},f^{*}T_{X}\otimes \mathcal{I}_{B})-h^{1}(\mathbb{P}^{1},f^{*}T_{X}\otimes \mathcal{I}_{B}).$$
If $H^{1}(\mathbb{P}^{1},f^{*}T_{X}\otimes \mathcal{I}_{B}) = 0$ then $\Mor_{l}(\p^{1},X,B)$ is smooth at $[f]$ and its tangent space is given by $T_{[f]}\Mor_{l}(\p^{1},X,B)\cong H^{0}(\mathbb{P}^{1},f^{*}T_{X}\otimes \mathcal{I}_{B})$.
By Riemann-Roch theorem we get $\dim_{[f]}\Mor_{l}(\p^{1},X,B)\geq \chi(\p^{1},f^{*}T_{X})-lg(B)\dim(X) = -K_{X}\cdot f_{*}\p^{1}+(1-lg(B))\dim(X)$.

Finally we briefly recall the notions of free and very free rational curve and their relationships with the concepts of uniruled and rationally connected variety.
\begin{Definition}
A rational curve $f:\p^{1}\rightarrow X$ is called \textit{free} if $f^{*}T_{X}$ is generated by global sections, and \textit{very free} if $f^{*}T_{X}$ is ample. 
\end{Definition}
It turns out that in characteristic zero, $X$ is \textit{uniruled} if and only if there exists on $X$ a free rational curve, and $X$ is \textit{rationally connected} if and only if there exists on $X$ a very free rational curve.
\subsubsection*{Conic Connectedness} 
Let us begin by giving a weak result on conic connectedness, whose proof is very simple and it prepares us for a stronger result.
\begin{Proposition}\label{CCS}
Let $X\subset\p^{N}$ be a variety set theoretically defined by homogeneous polynomials $G_{i}$ of degree $d_{i}$, for $i = 1,\ldots,m$. If 
$$\sum_{i=1}^{m}d_{i}\leq \frac{N}{2},$$
then $X$ is connected by singular conics.
\end{Proposition}
Before proving the theorem, let's observe the following fact.
\begin{Remark}\label{Mori}
From classical arguments of deformations of chains of rational curves we have that a singular conic through two general points on a smooth variety can be deformed into a smooth conic. However, the existence of a smooth conic $f:\p^{1}\rightarrow X$ through two general points on a projective variety does not imply the existence of a singular conic, this is true if $\dim_{[f]}(\Mor(\p^{1},X;f_{|\{0,\infty\}}))\geq 2$. This is Mori's Bend-and-Break lemma (\cite{De}, Proposition 3.2). Let us underline the fact that here we only ask for singular conics. 
\end{Remark}
\begin{proof}
Let $x\in X$ be a general point. Lines in $\mathbb{P}^{N}$ through $x$ are parametrized by $\mathbb{P}^{N-1}$. Forcing such a line to be contained in the hypersurface $\{G_{i} = 0\}$ gives $d_{i}$ equations for any $i = 1,...,m$. So for the dimension of the variety of lines $\mathcal{L}_{x}$ we have
$$\dim(\mathcal{L}_{x})\geq N-1-\sum_{i=1}^{m}d_{i}.$$  
Then when $\sum_{i=1}^{m}d_{i}\leq N-1$ the variety $X$ is covered by lines. Let $\mathfrak{Loc}_{x}$ be the locus described on $X$ by lines in $X$ through $x$. From now on we consider such locus in the ambient space $\p^{N}$ in order to have nice intersection properties. Moreover
$$\dim(\mathfrak{Loc}_{x})\geq N-\sum_{i=1}^{m}d_{i}.$$
Let $y\in X$ be another general point. Again we have $\dim(\mathcal{L}_{y})\geq N-1-\sum_{i=1}^{m}d_{i}$ and $\dim(\mathfrak{Loc}_{y})\geq N-\sum_{i=1}^{m}d_{i}$. Our numerical hypothesis yields 
$$\dim(\mathfrak{Loc}_{x})+\dim(\mathfrak{Loc}_{y})-N\geq 2(N-\sum_{i=1}^{m}d_{i})-N\geq 0.$$ 
Then $\mathfrak{Loc}_{x}\cap\mathfrak{Loc}_{y}\neq\emptyset$ and $x,y\in X$ can be connected by a singular conic.
\end{proof}
We are now ready to prove a stronger result. 
\begin{Theorem}\label{CC}
Let $X\subset\p^{N}$ be a variety set theoretically defined by homogeneous polynomials $G_{i}$ of degree $d_{i}$, for $i = 1,..,m$. If  
$$\sum_{i=1}^{m}d_{i}\leq \frac{N+m}{2}$$
then $X$ is connected by singular conics.
Assume $X$ to be smooth and the equations $G_{i}$'s to be scheme theoretical equations for $X$ and in decreasing order of degrees. If 
$$\sum_{i=1}^{c}d_{i}\leq \frac{N+c}{2},$$
where $c=N-n$, then $X$ is conic-connected by smooth conics also.
\end{Theorem}
\begin{proof}
Let $x,y\in X$ be two general points; we can assume $x = [1:0...:0]$ and $y = [0:...:0:1]$. We write $G_{i} = \sum_{j_{0}+\cdots+j_{N} = d_{i}}g_{j_{0},...,j_{N}}^{i}x_{0}^{j_{0}}...x_{N}^{j_{N}}$. Since $x,y$ lie in $X$ we get two conditions $G_{i}(1,0,...,0) = g_{d_i,0,...,0}^{i} = 0$ and $G_{i}(0,...,0,1) = g_{0,...,0,d_i}^{i} = 0$ for $i = 1,...,m$.

Let $p = [x_{0},...,x_{N}]$ be a point in $\p^{N}$. We parametrize lines through $x$ by $ux+vp = [u+vx_{0},...,vx_{N}]$, and lines through $y$ by $uy+vp = [vx_{0},...,u+vx_{N}]$. Now $G_{i}(ux+vp)$ is a polynomial of degree $d_{i}$ in $u,v$, it has $d_{i}+1$ coefficients, but the coefficients of $u^{d_i}$ does not appear because $x\in X$. So from $G_{i}(ux+vp)\equiv 0$ we get $d_{i}$ conditions and summing up on $i=1,...,m$ we have $\sum_{i=1}^{m}d_{i}$ equations, and we denote by $\mathfrak{Loc}_{x}$ the corresponding locus.
Similarly from $G_{i}(uy+vp)\equiv 0$ with $i=1,...,m$ we get $\sum_{i=1}^{m}d_{i}$ equations. Let $\mathfrak{Loc}_{y}$ be the locus of lines in $X$ through $y$. Note that the systems of equations defining $\mathfrak{Loc}_{x}$ and $\mathfrak{Loc}_{y}$ have $m$ common equations which are exactly $G_{1},...,G_{m}$, that can be found putting $u=0$ and $v=1$. So the intersection $\mathfrak{Loc}_{x}\cap \mathfrak{Loc}_{y}$ is defined by at most $2\sum_{i=1}^{m}d_{i}-m$ equations. Our numerical hypothesis ensures that this intersection is not empty.\\
Now assume $X\subseteq \p^{N}$ to be smooth and scheme theoretically defined by equations of degree $d_{1}\geq \cdots\geq d_{m}$. We use the same trick as in Theorem $2.4$ from \cite{IR}. By a result in \cite{BEL}, making a sort of liaison we can find $g_{i}\in H^{0}(\p^{N},\mathcal{I}_{X}(d_{i}))$ for $i = 1,...,c$ such that 
$$Y := Z(g_{1},...,g_{c}) = X\cup X^{'},$$
and $X^{'}$ intersects $X$ in a divisor when nonempty. If $x\in X$ is a general point then a line through $x$ is contained in $X$ if and only if it is contained in $Y$. This means that $\mathcal{L}_{x}(X)$ and $\mathcal{L}_{x}(Y)$ coincide set theoretically, and the same is true for the cones of lines through $x$. By the first part of the proof we have if $\sum_{i=1}^{c}d_{i}\leq \frac{N+c}{2}$ then there is a singular conic through two general points of $X$.
But we are now assuming $X$ to be smooth, and by general smoothing arguments (\cite{De}, Proposition 4.24) a singular conic through two general points $x,y$ can be deformed into a smooth conic containing $x,y$, so $X$ is conic-connected.
\end{proof}

We report an example to clarify the steps of our proof.

\begin{Example}
Consider the smooth quadric surface $X\subseteq \mathbb{P}^{3}$ defined by $G :=x_{0}x_{3}-x_{1}x_{2} = 0$, and the points $x = [1:0:0:0]$, $y = [0:0:0:1]$. From $G(ux+vp)\equiv 0$ and $G(uy+vp)\equiv 0$ we get 
$$
\left\{
\begin{array}{l}
x_{0}x_{3}-x_{1}x_{2} = 0;\\
x_{3} = 0;
\end{array}
\right.
\qquad
\left\{
\begin{array}{l}
x_{0}x_{3}-x_{1}x_{2} = 0;\\
x_{0} = 0;
\end{array}
\right.
$$
respectively. Computing their intersection we get two singular conics connecting $x$ and $y$, the conic $\{x_{2} = x_{3} = 0\}\cup \{x_{0} = x_{2} = 0\}$, and the conic $\{x_{1} = x_{3} = 0\}\cup \{x_{0} = x_{1} = 0\}$. 
\end{Example}

\begin{Remark}
In the range of Theorem \ref{CC} $X$ is covered by lines. The usual numerical condition to ensure that a variety is covered by lines is $\sum_{i=1}^{m}d_{i}< N$. Since $X$ is non degenerate $d_{i}\geq 2$ for any $i=1,...,m$. So under the numerical hypothesis of Theorem \ref{CC} we have $2m\leq \sum_{i=1}^{m}d_{i}\leq \frac{N+m}{2}$ which is equivalent to $3m\leq N$. In particular we get $m<N$ which implies 
$$\sum_{i=1}^{m}d_{i}\leq \frac{N+m}{2}< N.$$
The inequality $\sum_{i=1}^{m}d_{i}< N$ forces $X$ to be covered by lines.
\end{Remark}

In \cite{BH} \textit{Bonavero} and \textit{H\"oring} prove a similar fact using a different argument and taking $X$ to be a general scheme theoretical complete intersection. In the case $m = c$,  we get from Theorem \ref{CC} the following corollary, slightly weaker than theirs.

\begin{Corollary}\label{bonhor}
Let $X\subset\p^{N}$ be a smooth complete intersection defined by homogeneous polynomials $G_{i}$ of degree $d_{i}$, for $i = 1,..,c$. If 
$$\sum_{i=1}^{c}d_{i}\leq \frac{n}{2}+c,$$
then $X$ is conic-connected.
\end{Corollary} 

Furthermore, when the equality $\sum_{i=1}^{c}d_{i} = \frac{n+1}{2}+c$ holds, Bonavero and H\"oring prove that the number of conics in $X$ through two general points is finite, and they compute this number.

\begin{Remark}\label{SHA}
The inequality $\sum_{i=1}^{m}d_{i}\leq \frac{N+m}{2}$ is sharp. Let $X\subset \mathbb{P}^{4}$ be a smooth degree $d = 3$ hypersurface. Then $X$ is Fano of index $i_{X} = 2$, and $d = 3\leq N-1$ implies that $X$ is covered by lines. Since $d = 3\leq \frac{N+m+1}{2} = 3$ by the main result of \cite{BH} we have that $X$ is conic-connected. In general if $X$ is a smooth, covered by lines, conic-connected by smooth conics, Fano, projective variety of dimension $\dim(X) = n$, then $X$ is not connected by singular conics if and only if $i_{X} = \frac{n+1}{2}$. In our example $i_{X} = 2 = \frac{n+1}{2}$ and a proof of
this fact can be found in \cite{Wa}. The general cubic hypersurface in $\mathbb{P}^{4}$ is an example of a conic-connected variety which is not connected by singular conics and it is at the limit of our inequality.  
\end{Remark}

\begin{Remark}\label{smooth}
We want to highlight the role of the smoothness and of the singular conics in our argument.
\begin{itemize}
\item[-] Consider the cone over an elliptic cubic curve $X = Z(x_{0}x_{N}^{2} - x_{1}^{3} - x_{1}x_{0}^{2})\subset \mathbb{P}^{N}$, clearly $X$ is not ``smooth conic''-connected for any $N$. However two general points can be connected by a singular conic.
\item[-] Consider the rational normal scroll $X\subset\p^{4}$. It is conic-connected, but if one the two points is on the $(-1)$-curve on $X$ we actually get a singular conic but not a smooth one.   
\end{itemize}
\end{Remark}

\begin{Remark}
Suppose $X$ to be smooth. If $\sum_{i=1}^{c}d_{i}\leq \frac{N+c}{2}$, from $2c\leq \sum_{i=1}^{c}d_{i}\leq \frac{N+c}{2}$, we get $2c\leq n$. We are in the range of the Hartshorne Conjecture unless $X$ is quadratic. So if the  Hartshorne Conjecture is true \ref{CC} follows from the main theorem of \cite{BH}, and the case when $X$ is quadratic is covered by the main theorem of \cite{IR}.
\end{Remark}

\subsubsection*{Counting singular conics}
If $X$ is smooth and in Theorem \ref{CC} the equality holds we expect the number of singular conics through two general points to be finite.

A theorem by \textit{Faltings} in \cite{Fa} says that if $X$ is {\it smooth} and the number of equations {\it scheme theoretically} defining $X$ is small, precisely $m\leq \frac{N}{2}$, then $X$ is a complete intersection. Since $X$ is non-degenerate, $d_{i}\geq 2$ for any $i = 1,...,m$, so $\sum_{i=1}^{m}d_{i}\geq 2m$. From $2m\leq \sum_{i=1}^{m}d_{i}\leq \frac{N+m}{2}$ we get $m\leq \frac{N}{3}$. As $\frac{N}{3}< \frac{N}{2}$, a smooth variety scheme theoretically defined by the given equations and in the range of Theorem \ref{CC} is a complete intersection.

\begin{Proposition}
Let $X = Z(G_{1})\cap\cdots\cap Z(G_{c})\subset\p^{N}$ be a smooth complete intersection. If the equality $\sum_{i=1}^{c}d_{i} = \frac{N+c}{2}$ holds then the number of singular conics through two general points of $X$ is finite and is given by 
$$C_{2,2} = \prod_{i=1}^{c}d_{i}!(d_{i}-1)!$$
\end{Proposition}
\begin{proof}
We use the same notation of the Theorem \ref{CC}. Forcing $G_{i}(ux+vp) = 0$ we get $d_{i}$ equations of degrees $1,2,\ldots,d_{i}$ respectively, indeed we are forcing the vanishing of all terms of the dehomogenization of $G_{i}$ in an affine chart containing $x$ up to degree $d_{i}$. Similarly from $G_{i}(up+vy)$ we have $d_{i}-1$ equations of degrees $1,2,\ldots,d_{i}-1$; here we do not take the equation $G_{i}(p) = 0$ of degree $d_{i}$ because we have already considered it in the previous step. Summing up on $i=1,\ldots,c$ we have $2\sum_{i=1}^{c}d_{i}-c$ independent equations in $\p^{N}$. Since the equality $2\sum_{i=1}^{c}d_{i}-c = N$ holds their intersection consists of a finite number of points. By Bezout's theorem this number is given by the product of the degrees $d_{1}!(d_{1}-1)!\cdots  d_{c}!(d_{c}-1)! = \prod_{i=1}^{c}d_{i}!(d_{i}-1)!$
\end{proof}

\begin{Example}
Two points on a smooth quadric $X\subset\p^{3}$ can be connected by $2$ singular conics. A more interesting example is given by a complete intersection of two quadrics $X = X_{1}\cap X_{2}\subset\p^{6}$. In this case, our formula predicts the existence of $4$ singular conics. Indeed in this case $\mathfrak{Loc}_{x}$ is a surface of degree $4$ in the tangent space $T_{x}X$, so $\mathfrak{Loc}_{x}\cap T_{y}X$ consists of $4$ points.
\end{Example}

\subsubsection*{Acknowledgements}
The authors would like to express thanks for the wonderful environment created at \textit{P.R.A.G.MAT.I.C. (Catania, Italy, September 2010)} to the local organizers: \textit{Alfio Ragusa}, \textit{Giuseppe Zappal$\grave{a}$}, \textit{Rosario Strano}, \textit{Renato Maggioni} and \textit{Salvatore Giuffrida}, and they also acknowledge the support from the University of Catania. The authors heartily thank \textit{Prof. Paltin Ionescu} and \textit{Dr. Jos$\acute{e}$ Carlos Sierra} for the introduction to the subject, many helpful comments and suggestions.

\end{document}